\documentclass[12pt,oneside]{amsart}
\usepackage{amsfonts, amsmath, amssymb, amscd}
\usepackage{hyperref}
\hypersetup{
  colorlinks   = true, 
  urlcolor     = blue, 
  linkcolor    = blue, 
  citecolor   = red 
}
\usepackage{anysize}

\newtheorem{theorem}{Theorem}[section]
\newtheorem*{theorem*}{Theorem}

\newtheorem{claim}[theorem]{Claim}

\newtheorem{corollary}[theorem]{Corollary}

\theoremstyle{definition}

\newtheorem{problem}[theorem]{Problem}

\theoremstyle{remark}
\newtheorem{remark}[theorem]{Remark}

\numberwithin{equation}{section}

\DeclareMathOperator{\Res}{Res}
\DeclareMathOperator{\inte}{int}

\renewcommand{\epsilon}{\varepsilon}

\begin{document}

\title{Residues and the Combinatorial Nullstellensatz}

\author{Roman~Karasev}

\address{Dept. of Mathematics, Moscow Institute of Physics and Technology, Institutskiy per. 9, Dolgoprudny, Russia 141700}
\address{Institute for Information Transmission Problems RAS, Bolshoy Karetny per. 19, Moscow, Russia 127994}

\email{r\_n\_karasev@mail.ru}
\urladdr{http://www.rkarasev.ru/en/}

\thanks{Supported by the Dynasty Foundation.}





\begin{abstract}
We interpret the Combinatorial Nullstellensatz of Noga~Alon as a multidimensional residue formula, describe some consequences of this interpretation and related open problems.
\end{abstract}

\maketitle

\section{Introduction}

The Combinatorial Nullstellensatz of Noga~Alon~\cite{alon1999} turned out to be an efficient tool to obtain results in combinatorics and discrete geometry. This is an almost elementary algebraic statement, but it has not so elementary consequences and generalizations.

In the recent papers~\cite{lason2010,karpet2011} a version of the Combinatorial Nullstellensatz was expressed as a certain formula, which turned out to be useful in several problems (see~\cite{kn2012,klw2012}, for example):

\begin{theorem}[The Combinatorial Nullstellensatz]
\label{cn-prod}
Suppose a multivariate polynomial $f(x_1,x_2,\dots,x_n)$ over a field $\mathbb{F}$ has degree at most $c_1+c_2+\dots+c_n$, where $c_i$ are non-negative integers. Denote by $C$ the coefficient at $x_1^{c_1}\dots x_n^{c_n}$ in $f$. Let $A_1$, $A_2$, \ldots, $A_n$ be arbitrary subsets of $\mathbb{F}$ such that $|A_i|=c_i+1$ for any $i$. Put $\varphi_i(x)=\prod_{\alpha\in A_i}(x-\alpha)$. Then
\begin{equation}
\label{cn-eq}
C=\sum_{\alpha_i\in A_i} \frac{f(\alpha_1,\dots,\alpha_n)}
{\varphi_1'(\alpha_1)\dots \varphi_n'(\alpha_n)}.
\end{equation}
In particular, if $C\ne 0$, then there exists a system of representatives $\alpha_i\in A_i$ such that $f(\alpha_1,\alpha_2,\dots,\alpha_n)\ne 0$.
\end{theorem}

The general way to apply this theorem, developed by Fedor Petrov in~\cite{karpet2011}, is as follows: Express a combinatorial statement in the from that a certain polynomial $f$ of appropriate degree attains a nonzero value on the product $A_1\times \dots \times A_n$. In order to prove this, by the theorem, we need to show that $C$ is nonzero. Then we try to modify the polynomial $f$ without changing $C$, usually it corresponds to a special choice of the parameters of the initial combinatorial problem, and obtain another polynomial $g$ such that the right hand side of (\ref{cn-eq}) contains one (or slightly more) summands that are easy to calculate.

In~\cite{karpet2011} a simple proof of this theorem was given, using the Lagrange interpolation formula, see the review~\cite{gassa2000} for more information about interpolation. 

The emphasis of this note is that this formula can be viewed, less elementary, as a multidimensional residue formula. In what follows we explain the meaning of this and try to show other situations when this point of view may be useful. In principle, this allows, with some care, to consider the case when the sets $A_i$ are multisets (sets with some multiplicities). We also show the relation between the Combinatorial Nullstellensatz and the old Cayley--Bacharach theorem about incidence of intersection of hypersufraces.

\section{Residues on the product of projective lines}

Let us interpret the Combinatorial Nullstellensatz (Theorem~\ref{cn-prod}) as a corollary of the residue theorem~\cite[Ch.~5, \S~1]{gh1978}: 

\begin{theorem}[The residue theorem]
\label{residue-theorem}
Let $D_1,\ldots, D_n$ be a set of divisors on a compact analytic $n$-dimensional manifold $M$, having a zero-dimensional intersection. Then for any holomorphic $\omega\in \Omega^n(M\setminus \bigcup_{i=1}^n D_i)$ we have:
$$
\sum_{x\in D_1\cap\dots\cap D_n} \Res_x \omega = 0.
$$
\end{theorem}

\begin{remark}
Note that the value $\Res_x \omega$ actually depends on the set of divisors $(D_1,\ldots, D_n)$. In particular it changes sign if the divisors are permuted by an odd permutation. To keep the things clear, we restrict ourselves to ``geometric'' divisors, that is combinations of prime divisors with unit coefficients.
\end{remark}

\begin{remark}
The algebraic version of Theorem~\ref{residue-theorem} is valid for any algebraically closed field of coefficients, but let us restrict ourselves to $\mathbb C$ here.
\end{remark}

Now we deduce Theorem~\ref{cn-prod} from the residue theorem. Take the product of projective lines $M=\underbrace{\mathbb CP^1\times\dots\times \mathbb CP^1}_n$.  Consider a grid subset:
$$
X = X_1\times\dots \times X_n \subseteq \underbrace{\mathbb C\times\dots\times \mathbb C}_n,
$$
where $|X_i|=k_i$, and a polynomial $f\in \mathbb C[z_1,\ldots,z_n]$. Denote 
$$
g_i(z) = \prod_{x\in X_i} (z-x),
$$
and apply the residue theorem to the differential form
$$
\omega = \frac{f(z_1,\ldots, z_n)dz_1\wedge\dots\wedge dz_n}{g_1(z_1)\dots g_n(z_n)}.
$$
The singularities of this differential form are at sets 
$$
D_i = \{(z_1,\ldots, z_n)\in (\mathbb CP^1)^{\times n} : z_i\in X_i\ \text{or}\ z_i = \infty\},
$$
that we consider as divisors. The intersection of these divisors is 
$$
D_1\cap\dots\cap D_n = (X_1\cup\{\infty\})\times\dots\times (X_n\cup\{\infty\}),
$$
and applying the residue formula yields:
\begin{equation}
\label{res-form}
\sum_{(z_1,\ldots, z_n)\in D_1\cap\dots\cap D_n} \Res_{(z_1,\ldots, z_n)} \binom{\omega}{D_1 D_2 \dots D_n} = 0.
\end{equation}
The residue at $(\infty, \ldots, \infty)$ with respect to $t_1=\frac{1}{z_1}, \ldots, t_n=\frac{1}{z_n}$ is 
\begin{multline*}
\Res_{(\infty,\ldots,\infty)} \omega = (-1)^n\Res_{(0,\ldots, 0)} \frac{f(\frac{1}{t_1},\ldots, \frac{1}{t_n})dt_1\wedge\dots\wedge dt_n}{t_1^2g_1(\frac{1}{t_1})\dots t_n^2g_n(\frac{1}{t_n})} = \\
= (-1)^n\Res_{(0,\ldots, 0)} f\left(\frac{1}{t_1},\ldots, \frac{1}{t_n}\right)dt_1\wedge\dots\wedge dt_n \prod_{i=1}^n\left(t_i^{n_i-2}\prod_{x\in X_i} \frac{1}{1 - t_ix})\right),
\end{multline*}
if the total degree $\deg f \le \sum_{i=1}^n (k_i-1)$, then we have a simple formula
$$
\Res_{(\infty,\ldots,\infty)} \binom{\omega}{D_1 D_2 \dots D_n} = (-1)^n c_{k_1-1,\ldots, k_n-1},
$$
where $c_{k_1-1,\ldots, k_n-1}$ is a coefficient at $z_1^{k_1-1}\dots z_n^{k_n-1}$ in $f(z_1,\ldots, z_n)$.

The equation~\ref{res-form} would give the desired result (up to sign), but the intersection $D_1\cap\dots\cap D_n$ has points with some coordinates $\infty$, and some finite. Fortunately, this issue is resolved by considering ``rearranged'' divisors
$$
D'_i = \{(z_1,\ldots, z_n)\in (\mathbb CP^1)^{\times n} : z_i\in X_i\ \text{or}\ z_{i+1} = \infty\},
$$
where the indexes of $z_i$ are taken modulo $n$. Now the intersection of divisors becomes what we want:
$$
D'_1\cap\dots\cap D'_n = (X_1\times\dots\times X_n)\cup\{(\infty,\infty,\ldots,\infty)\},
$$
and the result follows. Because of the cyclic shift the residue at infinity becomes $-c_{k_1-1,\ldots, k_n-1}$ and the formula is correct.

\section{Observation on the residue formula for two sets of divisors}

The trick with rearranging the divisors may be replaced by the following version of the residue formula:

\begin{theorem}[Gelfond--Khovanskii, 2002]
\label{residue-theorem2}
Let $D_1,\ldots, D_n$ and $D'_1,\ldots, D'_n$ be two sets of divisors on a compact analytic $n$-dimensional manifold $M$, each having a zero-dimensional intersection. Assume $D_i\cap D'_i=\emptyset$ for every $i$ and put $Z = \bigcup_{i=1}^n D_i \cup \bigcup_{i=1}^n D'_i$. Then for any holomorphic $\omega\in \Omega^n(M\setminus Z)$ we have:
$$
\sum_{p\in D_1\cap\dots\cap D_n} \Res_p \omega = (-1)^n \sum_{q\in D'_1\cap\dots\cap D'_n} \Res_q \omega.
$$
Here the residues on the left are considered with respect to the set of divisors $(D_1,\ldots, D_n)$ and the residues on the right use the set of divisors $(D'_1,\ldots, D'_n)$.
\end{theorem}

This theorem in the analytic case was established by Gelfond and Khovanskii in~\cite[Theorem~2]{gekh2002}. The algebraic analogue of this theorem for algebraically closed fields follows from the ordinary residue theorem (like Theorem~\ref{residue-theorem}) by the same rearranging trick: put $D''_i = D_i\cup D'_{i+1}$ (the indices understood mod $n$) and note that the points of the intersection $D''_1\cap\dots\cap D''_n$ split into two subsets $D_1\cap\dots\cap D_n$ and $D'_1\cap\dots\cap D'_n$.

In view of Theorem~\ref{residue-theorem2} the Combinatorial Nullstellensatz is easily obtained by taking $M=\underbrace{\mathbb CP^1\times\dots\times \mathbb CP^1}_n$,
$$
\omega = \frac{f(z_1,\ldots, z_n)dz_1\wedge\dots\wedge dz_n}{g_1(z_1)\dots g_n(z_n)},
$$
$$
D_i = \{(z_1,\ldots, z_n) : z_i \in X_i\}\quad\text{and}\quad D'_i=\{(z_1,\ldots, z_n) : z_i =\infty\}.
$$
Again, the sum of residues at finite points turns out to be equal up to sign to the residue at $(\infty,\ldots, \infty)$. 

\begin{remark}
Another observation is that the assumption that the degree of $f$ is at most $c_1+c_2+\dots+c_n$ in Theorem~\ref{cn-prod} is not actually needed. What is really needed is that besides the monomial $C z_1^{c_1}\dots z_n^{c_n}$ all other monomials $C' z_1^{d_1}\dots z_n^{d_n}$ of $f$ have $d_i<c_i$ for \emph{at least one} index $i$.
\end{remark}

\section{Toric version of the Combinatorial Nullstellensatz}

Continuing to follow the results of~\cite{gekh2002}, we observe that Theorem~\ref{cn-prod} can also be viewed as a very particular case of the toric residue formula in~\cite{gekh2002}. Let us show this in more detail. We are going to deal with Laurent polynomials $f\in \mathbb C[z_1,\ldots, z_n,z_1^{-1},\ldots, z_n^{-1}]$ and their Newton polytopes $N(f)\in\mathbb Z^n$, that is convex hulls of the degrees of nonzero monomials in $f$.

Take some $n$ Laurent polynomials $g_1,\ldots, g_n$, the Newton polynomial of their product $N(g_1\dots g_n)$ equals the Minkowski sum $N(g_1)+\dots + N(g_n)$. Following Gelfond and Khovanskii~\cite{gekh2002} the system $N(g_1), \ldots, N(g_n)$ is called \emph{unfolded} if in their Minkowski sum every face $F$ of positive codimension in its unique decomposition $F=F_1+\ldots+F_n$ into the sum of faces of the polytopes $N(g_1), \ldots, N(g_n)$ has at least one zero-dimensional $F_i$. This is a certain requirement of general position and it is easy to check that for polynomials $g_1(z_1),\ldots,g_n(z_n)$ each depending on its respective one variable (as in Theorem~\ref{cn-prod}) this assumption is satisfied.

\begin{theorem}[Gelfond--Khovanskii, 2002]
Consider a differential form 
$$
\omega = \frac{f}{g_1\dots g_n} dz_1\wedge\dots\wedge dz_n
$$
where the system of Newton polytopes $N(g_1), \ldots, N(g_n)$ is unfolded. Let $Z$ be the set of common zeros of $g_1,\ldots, g_n$ in $T={\mathbb C^*}^n$ and $V$ be the set of vertices of the polytope $N=N(g_1)+\dots + N(g_n)$, then 
\begin{equation}
\label{toric-res-eqn}
\sum_{z\in Z} \Res_z \omega = (-1)^n \sum_{v\in V} k_v \Res_v \omega.
\end{equation}
\end{theorem}

Definitely, this formula needs some explanations. The left hand side of (\ref{toric-res-eqn}) is the ordinary sum of residues over the ``finite'' points of the toric variety, that is points lying in $T$. The right hand side is the sum of residues in the ``infinite'' points of the compactification of $T$ that gives the toric variety. The coefficients $k_v$ are integers depending on the combinatorial structure of $N$ near its vertex $v$ and the value $\Res_v \omega$ is calculated explicitly as the constant term in the Laurent series expansion of $\frac{z_1^{v_1}\dots z_n^{v_n}}{g_1\dots g_n}$ multiplied by the Laurent polynomial $\frac{f}{z_1^{v_1-1}\dots z_n^{v_n-1}}$. Here it is convenient to denote $z_1^{v_1}\dots z_n^{v_n}$ by $z^v$ for $z=(z_1,\ldots, z_n)$ and $v=(v_1,\ldots, v_n)$; also denote by $\mathbf{e}=(1,\ldots, 1)\in \mathbb Z^n$ the all-unit vector.

Let us describe a particular case when everything has a very explicit form.  We make the important assumption: For every vertex $v\in V$ there exists an outer support halfspace $H$ to $N$ at $v$ such that $H\cap N = \{v\}$ and the polytope $N(f)+\mathbf{e}$ does not intersect $\inte H$ (in Theorem~\ref{cn-prod} this corresponds to the degree upper bound).

The set of vertices of $N$ thus splits into two parts $V=V_+\cup V_0$ depending on whether they are outside $N(f) + \mathbf{e}$ or on its boundary. It is easy to see that for $v\in V_+$ the value $\Res_v \omega$ is zero and for $v\in V_0$ it equals the coefficient at $z^{v-\mathbf{e}}$ in $f$, divided by the product of coefficients in $g_i$ at the monomials corresponding to the unique representation of $v$ as a sum of vertices of $N(g_1), \ldots, N(g_n)$. Finally we obtain:

\begin{corollary}
\label{toric-cn}
Under above assumptions $\sum_{z\in Z} \Res_z \omega$ equals a linear combination of the coefficients of $f$ at monomials $z^{v-\mathbf{e}}$ for $v\in V_0$ with integer coefficients $k_v$.
\end{corollary}

\begin{remark}
In some cases one may guarantee that the coefficient $k_v$ for $v\in V_0$ is nonzero. For example, this is the case when exactly $n$ facets of $N$ meet at $v$. It is easy to check that this is the case in Theorem~\ref{cn-prod}.
\end{remark}

\begin{remark}
As it was already discussed, when all zeros in $Z$ are simple then on the left hand side of (\ref{toric-res-eqn}) we have a sum of values of $f$ in the points of $Z$ with certain nonzero coefficients.
\end{remark}

\begin{remark}
Corollary~\ref{toric-cn} formally requires the points of $Z$ to have only nonzero coordinates (they have to lie in $T$), but it is easy to see that Theorem~\ref{cn-prod} follows in its full generality by a translation of the sets $A_i$ so that they avoid zero.
\end{remark}

\section{Residues on $\mathbb CP^n$ and the Cayley--Bacharach theorem}

Another version of the proof for Combinatorial Nullstellensatz arises if we consider the form
$$
\omega = \frac{f(z_1,\ldots, z_n)dz_1\wedge\dots\wedge dz_n}{g_1(z_1)\dots g_n(z_n)}.
$$
over the projective space $\mathbb CP^n$. Compared to the previous section, this approach allows to make the results more flexible and independent on the Newton polynomials of $f$ and $g_i$.

Suppose first that $\deg f \le \sum_{i=1}^n k_i - n - 1$. In this case a simple calculation shows that $\omega$ has no singularity over the hyperplane at infinity, and we obtain the equality (the residues are with respect to the divisors corresponding to $g_1, \ldots, g_n$)
$$
\sum_{(z_1,\ldots, z_n)\in X_1\times\dots X_n} \Res_{(z_1,\ldots, z_n)} \omega = \sum_{(z_1,\ldots, z_n)\in X_1\times\dots X_n} \frac{f(z_1,\ldots, z_n)}{g'_1(z_1)\dots g'_n(z_n)} = 0,
$$
which leads to the Cayley--Bacharach theorem (see~\cite{bach1886,cayley1889} and the textbook~\cite[Ch.~5, \S~2]{gh1978}): If $f$ is zero at all but one points of $X$, then it should be zero at the remaining point. We give here the general statement of the Cayley--Bacharach theorem: 

\begin{theorem}[Cayley--Bacharach, XIXth century]
\label{cb}
If the system of equations
\begin{eqnarray*}
g_1(x) &=& 0\\
&\ldots&\\
g_n(x) &=& 0
\end{eqnarray*}
of degrees $k_1,\ldots, k_n$ has $k=k_1k_2\dots k_n$ isolated solutions $X=\{x_1,\ldots, x_k\}$, then there exists a linear dependence with nonzero coefficients:
\begin{equation}
\label{cb-rel}
\sum_{i=1}^k \alpha_i f(x_i) = 0
\end{equation}
between values of every polynomial of degree $\deg f \le \sum_{i=1}^n k_i - n - 1$. In particular, the polynomial should be zero on $X$ if and only if it is zero on all but one points of $X$.
\end{theorem}

This theorem holds over arbitrary field if all the points of $X$ are defined over this field. Let us list some recent nontrivial uses of this theorem:

\begin{itemize}
\item 
An interesting application of the Cayley--Bacharach relations (\ref{cb-rel}) is distinguishing between nonnegative polynomials and sums of squares, see~\cite{ble2010} for further details.

\item 
The least nontrivial case of the Caylet--Bacharach theorem, for intersection of two triples of lines, was used in the recent paper~\cite{greentao2012} about Sylvester type problems.
\end{itemize}

It is curious that different particular cases of the Cayley--Bacharach theorem have their own names. For example, Miquel's six circle theorem~\cite{wikimiq} asserts that if $7$ out of $8$ vertices of a combinatorial cube $C$ in $\mathbb R^3$ lie on a quadratic surface $S$ then the remaining vertex of $C$ also must lie on $S$. Another particular case of the Cayley--Bacharach theorem is the result about cutting the integer points in a cube by hyperplanes (see~\cite[Theorem~6.3]{alon1999} and Problem~6 at \href{http://www.imo-official.org/}{IMO}~2007), which we state in a bit more general, than usual, form here:

\begin{corollary}
\label{cover-hyperpl}
Suppose we have $n$ families of hyperplanes $\mathcal H_1,\ldots,\mathcal H_n$ in $\mathbb CP^n$ with respective cardinalities $k_1,\ldots, k_n$. Define the intersection set
$$
X=\{H_1\cap\dots\cap H_n : H_1\in\mathcal H_1,\ldots, H_n\in\mathcal H_n\}
$$
and assume that it is discrete and has the maximum possible cardinality $k=k_1k_2\dots k_n$. If $x\in X$ is any point, then the set $X\setminus x$ cannot be covered by less than $\sum_{i=1}^n k_i - n$ hyperplanes that do not pass through $x$.
\end{corollary} 

Using the projective duality we obtain another consequence:

\begin{corollary}
Let $n$ finite point sets $X_1,\ldots, X_n\subset \mathbb CP^n$ have cardinalities $k_1,\ldots, k_n$. Assume that any system of representatives $x_i\in X_i$ defines a unique hyperplane $H(x_1,\ldots, x_n)$ containing $\{x_i\}_{i=1}^n$ and all these hyperplanes are distinct. Then one needs at least $\sum_{i=1}^n k_i - n$ points to pierce all such hyperplanes $H(x_1,\ldots, x_n)$ but one $H(x^0_1, \ldots, x^0_n)$ without touching this one.
\end{corollary}

Now return to the original statement of the Combinatorial Nullstellensatz, where $\deg f = \sum_{i=1}^n k_i - n$. In this case $\omega$ has the singularity at the hyperplane at infinity, and we should include this hyperplane to a divisor in the definition of the residues. Finite singularity hyperplanes are 
$$
\mathcal H_i = \{H : H=\{z_i = x\},\ x\in X_i\}.
$$
The hyperplane at infinity can be added to the first family of hyperplanes $\mathcal H_1$ for example, to give $\mathcal H_1^*$. The corresponding set
$$
X^*=\{H_1\cap\dots\cap H_n : H_1\in\mathcal H_1^*,H_2\in\mathcal H_2,\ldots, H_n\in\mathcal H_n\}
$$
will contain all the points of $X$, and the point $x^*$ at the infinite direction of $(1,0,\ldots,0)$ axis. Note that the form $\omega$ has a bad singularity in $x^*$, and the residue formula is hard to apply at this point. But this can be corrected, if we perturb the families $\mathcal H_i$ ($i=2,\ldots, n$) so that the point $x^*$ becomes a set of $k_2\dots k_n$ points with simple singularities, lying on the hyperplane at infinity. For these points the formula can be proved by induction, by putting the sum of residues to the hyperplane at infinity and applying the inductive assumption.

This proof is good, but it is much longer than the original proof without residues. In order to justify this we may generalize the Combinatorial Nullstellensatz in some way, for example:

\begin{theorem}
Suppose we have $n$ hypersurfaces $S_1, \ldots, S_n\subset \mathbb C^n$ with respective degrees $k_1,\ldots, k_n$, and their equations have the form 
$$
g_i(z_1,\ldots, z_n) = z_i^{k_i} + \text{terms of less degree}.
$$
Assume that they intersect in a discrete set $X$ of cardinality $k=k_1k_2\dots k_n$. If a polynomial $f(z_1,\ldots, z_n)$ has degree $\le \sum_{i=1}^n k_i - n$ and a nonzero coefficient at $z_1^{k_1-1}\dots z_n^{k_n-1}$, then its zero set cannot contain $X$.
\end{theorem}

It seems that for arbitrary $g_i(z_1,\ldots, z_n)$ the condition ``coefficient at $z_1^{k_1-1}\dots z_n^{k_n-1}$ is nonzero'' should be replaced by some other condition, depending on the maximal degree parts of $f, g_1,\ldots, g_n$.

\section{Further similar problems}

The first question is: Does the two-dimensional case of Theorem~\ref{cover-hyperpl} admit a simpler proof? Its elementary statement reads as follows: 

\begin{problem}
Suppose $n$ red and $m$ blues lines in the plane have $nm$ points of transversal red-blue intersection, denote this intersection set by $X$. Prove that if a family of green lines covers all points of $X$ but one then there are at least $n+m-2$ green lines.
\end{problem}

Another question is related to some algebraic constructions of hypergraphs in~\cite{fr2011}. We believe that the residues may help to answer it, but cannot tell anything particular at the moment.

\begin{problem}
\label{grid-lines}
Suppose $n$ red and $n$ blue lines in the plane have $n^2$ points of transversal red-blue intersection, again denote this intersection set by $X$. Describe all cases when $X$ can be covered by $n$ green lines, distinct from the original blue and red lines.
\end{problem}

There are nontrivial examples for Problem~\ref{grid-lines}: In $\mathbb F_p\times \mathbb F_p$ we may consider all vertical lines red, all horizontal lines blue, and all lines with a fixed slope green. Here $n=p$ is the characteristic of the field.

Another example is: Let $U\subset\mathbb F^*$ be a finite multiplicative subgroup of order $n$, which necessarily coincides with the $n$-th roots of unity. Consider the blue lines $\{x-uy\}_{u\in U}$, the red lines $\{y = u\}_{u\in U}$, and the green lines $\{x=u\}_{u\in U}$. This is a valid configuration in Problem~\ref{grid-lines} and an important observation is that all three color families of lines are concurrent. 

Actually, the case of interest in~\cite{fr2011} is when $n<p$ (in $\mathbb F_p$) and the green lines form the (concurrent) family of vertical lines $x=0, x=1,\ldots, x = n-1$. In~\cite[Lemma~2.9]{fr2011} it is shown that no such configurations (with vertical green lines) exists for $n>3$ over the field $\mathbb R$, the proof using combinatorics of pseudolines. The case of finite characteristic with this selection of vertical lines is reduced to the real case (see~\cite{fr2011}) for $p>n^{4n}$ using the Dirichlet theorem on approximation by rational numbers.

We have a couple of observations on Problem~\ref{grid-lines}, with no use of residues, considering concurrent families of lines:

\begin{claim}
\label{grid-conc}
In terms of Problem~\ref{grid-lines}, let $r_i(x)=0$ be the equations of the red lines, let $b_i(x)=0$ and $g_i(x)=0$ be the equations of blues and green lines respectively. If all the green lines are concurrent then there is a linear dependence between the products $R(x) = \prod_i r_i(x)$, $B(x) = \prod_i b_i(x)$, and $G(x) = \prod_i g_i(x)$.
\end{claim}

\begin{proof}
We denote by the same letter the straight line and its corresponding linear function. Let $x_0$ be the common point of the green lines. Note that on every line $g_i$ there must be at most $n$ points of $X$, because it meets at most $n$ red lines. Hence every $g_i$ contains exactly $n$ points of $X$ and these $n$-tuples are pairwise disjoint. Hence the common point $x_0$ cannot be in $X$.

Now choose coefficients $\alpha$ and $\beta$ so that $Z(x) = \alpha R(x) + \beta B(x)$ vanishes on $x_0$; it also vanishes on $X$. For every line $g_i$ the function $Z(x)$ vanishes on $g_i$ at $x_0$ and at $n$ intersection points $X\cap g_i$. Since $Z(x)$ has degree $\le n$ it must vanish on every $g_i$ and therefore it must be proportional to the product $G(x)$.
\end{proof}

\begin{claim}
\label{grid-conc2}
If we assume in Problem~\ref{grid-lines} that the red lines are concurrent and the green lines are concurrent, and also assume that $n$ is coprime with the characteristic of $\mathbb F$, then the example with roots of unity becomes unique up to projective transformation.
\end{claim}

\begin{proof}
After a projective transformation we assume that the red lines are $\{x = u\}_{u\in U}$ and the blue lines are $\{y = v\}_{v\in V}$. Then every green line $g_i$ is a graph of a linear bijection $U\to V$. Hence we have a set of linear transforms $g_j^{-1}g_i$ for the set $U$. These linear transforms must preserve the mass center $\frac{1}{n}(u_1+\dots + u_n)$ of $U$, and after another shift of the coordinates we assume that this mass center is zero and all the transforms $g_j^{-1}g_i$ are multiplications by a constant $c_{j i}$. Let us also rescale so that $U$ contains $1$. Then every $c_{j i}$ is contained in $U$, and since there must be at least $n$ distinct constants corresponding to $g_1^{-1}g_1, g_2^{-1}g_1, \ldots, g_n^{-1}g_1$ then we conclude that $U$ is a multiplicative subgroup and the transforms are multiplications by elements of this groups. After an appropriate shift and rescaling of the $y$ axis the set $V$ becomes equal to $U$.
\end{proof}

Finally we mention a problem from~\cite{gassa2000} related to the polynomial interpolation, which is in the spirit of the present discussion:

\begin{problem}
Suppose $X$ is a set of $\binom{n+2}{2}$ points in the plane such that for any $x\in X$ there exist $n$ lines covering $X\setminus \{x\}$ and not touching $x$. Describe such sets $X$ or, at least, prove that some $n+1$ points of $X$ lie on a single line.
\end{problem}

\bibliography{../bib/karasev}
\bibliographystyle{abbrv}
\end{document}